\documentclass[12pt]{amsart}
\usepackage{hyperref, tikz}
\title[trees and rational generating functions]{Equivalence classes of nodes in trees\\and rational generating functions}
\author{Amritanshu Prasad}
\address{The Institute of Mathematical Sciences, Chennai.}
\email{amri@imsc.res.in}
\date{\today}

\newcommand{\B}{\mathbf B}
\newcommand{\Fq}{\mathbf F_q}
\newcommand{\one}{\mathbf 1'}
\newcommand{\n}{\mathbf n}
\newcommand{\m}{\mathbf m}
\newcommand{\X}{\mathbf X}

\newtheorem*{definition*}{Definition}
\newtheorem*{conjecture*}{Conjecture}
\newtheorem{theorem}{Theorem}

\theoremstyle{remark}
\newtheorem{example}{Example}[theorem]

\begin{document}
\maketitle

\begin{abstract}
  Let $c_n$ denote the number of nodes at a distance $n$ from the root of a rooted tree.
  A criterion for proving the rationality and computing the rational generating function of the sequence $\{c_n\}$ is described.
  This criterion is applied to counting the number of conjugacy classes of commuting tuples in finite groups and the number of isomorphism classes of representations of polynomial algebras over finite fields.
  The method for computing the rational generating functions, when applied to the study of point configurations in finite sets, gives rise to some classical combinatorial results on Bell numbers and Stirling numbers of the second kind.
  When applied to the study of vector configurations in a finite vector space, it reveals a connection between counting such configurations and Gaussian binomial coefficients.
\end{abstract}

\section{Introduction}
\label{sec:intro}
This paper begins by describing a technique for proving the rationality of, and often explicitly computing, ordinary generating functions of certain combinatorial sequences (Theorem~\ref{theorem:lineal}).
It applies to sequences whose $n$th term can be expressed as a the number of nodes in a rooted tree at a distance $n$ from the root.
The rationality rests upon the finiteness of the number of what are called lineal isomorphism classes of nodes.

The counting of simultaneous conjugacy classes of commuting $n$-tuples in a finite group $G$ is, in general, a difficult combinatorial problem.
However, the rationality of the generating function associated to this count turns out to be an easy consequence of Theorem~\ref{theorem:lineal} (see Theorem~\ref{theorem:group-tuples}).
These generating functions are computed explicitly for the first five symmetric groups (Table~\ref{tab:hSn}).

A slight variant of this result also shows that if $c_{q,m}(n)$ denotes the number of isomorphism classes of $m$-dimensional representations of the polynomial algebra $\Fq[x_1,\dotsc,x_n]$ (here $\Fq$ is a finite field of order $q$), then $c_{q,m}(n)$ (as a sequence in $n$) has a rational generating function (Theorem~\ref{theorem:polyalg}).

The method from Theorem~\ref{theorem:lineal} for computing generating functions can sometimes be applied advantageously even to situations where rationality is easy to see by other methods.
For instance, when applied to counting point configurations in finite sets, it leads to beautiful classical results concerning Bell numbers and Stirling numbers of the second kind.
When applied to counting vector configurations in finite vector spaces, it leads to the discovery of a new interpretation of Gaussian binomial coefficients.

\section{Lineal equivalence and rational generating functions}
\label{sec:line-equiv-rati}
Let $T$ denote the vertex set of a rooted tree with root vertex $x_0$.
Let $T_n$ denote the set of vertices of $T$ which are a distance $n$ from $x_0$.
Then $T$ is a disjoint union:
\begin{equation*}
  T = \coprod_{n=0}^\infty T_n.
\end{equation*}
We will give a sufficient condition for the formal generating function
\begin{equation}
  \label{eq:gen-fun}
  f_T(t) = \sum_{n=0}^\infty |T_n|t^n
\end{equation}
to be a rational function in $t$ and a technique for its computation.

If $X\in T_n$ and $Y\in T_{n+1}$ are connected by an edge, then we say that $Y$ is a \emph{child} of $X$, and write $X\to Y$.
More generally, if $X\in T_n$ and $Y\in T_{n+k}$ for some $k\geq 0$ are such that there exists a sequence
\begin{equation*}
  X=X_0\to X_1\to \dotsb \to X_k = Y,
\end{equation*}
then we say that $Y$ is a \emph{descendant} of $X$ (under our definition $X$ is a descendant of $X$).

For each $X\in T$, let $T(X)$ denote the full subtree consisting of the descendants of $X$.
This is again a rooted tree, with root $X$.
\begin{definition*}
  [Lineal Isomorphism]
  Two vertices $X$ and $Y$ of $T$ are said to be lineally isomorphic if the rooted trees $T(X)$ and $T(Y)$ are isomorphic (in other words, there is a graph isomorphism $T(X)\to T(Y)$ taking $X$ to $Y$).
\end{definition*}
Clearly, lineal isomorphism is an equivalence relation on $T$.
The equivalence classes of this relation are called \emph{lineal isomorphism classes}.
\begin{theorem}
  If $X$ and $Y$ are lineally isomorphic nodes in a rooted tree $T$, then for any lineal isomorphism class $C$, the number of children of $X$ in $C$ is equal to the number of children of $Y$ in $C$.
\end{theorem}
\begin{proof}
  Since $X$ and $Y$ are lineally isomorphic, there exists an isomorphism $T(X)\to T(Y)$ of rooted trees.
  This isomorphism defines a bijection from the children of $X$ in $C$ to the children of $Y$ in $C$.
\end{proof}
\begin{theorem}
  \label{theorem:lineal}
  Let $T$ be a rooted tree with finitely many lineal isomorphism classes $C_1,\dotsc,C_N$, the root of of $T$ lying in $C_1$.
  Let $\B = (b_{ij})$ be the $N\times N$ matrix where $b_{ij}$ is the number of children that a node in the class $C_j$ has in the class $C_i$.
  Then, for each $i\in \{1,\dotsc,N\}$,
  \begin{equation}
    \label{eq:ogfi}
    \sum_{n=0}^\infty |T_n\cap C_i|t^n = e_i'(I-\B t)^{-1} e_1.
  \end{equation}
  Here, for each $i\in \{1,\dotsc,N\}$, $e_i$ denotes the $i$th coordinate vector, viewed as an $N\times 1$ matrix, and $e_i'$ its transpose.
  In particular, the sequence $\{|T_n\cap C_i|\}_{n=0}^\infty$ has a rational generating function for each $i$.
  Consequently,
  \begin{equation}
    \label{eq:ogf}
    \sum_{n=0}^\infty |T_n|t^n = \one (I - \B t)^{-1}e_1.
  \end{equation}
  Here $\one$ is the $1\times N$ all-ones row vector.
\end{theorem}
\begin{proof}
  Let $v_j^{(n)} = |C_j\cap T_n|$ for $n\geq 0$ and $1\leq j\leq N$.
  Let $v^{(n)}$ denote the column vector with coordinates $(v^{(n)}_1,\dotsc, v^{(n)}_N)$.
  The hypothesis that the root of $T$ lies in $C_1$ implies that $v^{(0)} = e_1$.

  Since each node in $C_j\cap T_{n-1}$ contributes $b_{ij}$ elements to $C_i\cap T_n$,
  \begin{equation}
    \label{eq:recurrence}
    v_i^{(n)} = |C_i\cap T_n| = \sum_{j=1}^N |C_j\cap T_{n-1}|b_{ij} = \sum_{j=1}^N b_{ij}v_j^{(n-1)}.
  \end{equation}
  The recurrence relation (\ref{eq:recurrence}) can be written in matrix form as:
  \begin{equation*}
    v^{(n)} = \B v^{(n-1)},
  \end{equation*}
  upon iterating which (and using the fact that $v^{(0)} = e_1$), we get
  \begin{equation*}
    v^{(n)} = \B^ne_1.
  \end{equation*}
  Therefore,
  \begin{align*}
    \sum_{n=0}^\infty v^{(n)}t^n & = \sum_{n = 0}^\infty \B^nt^ne_1\\
    & = (I-\B t)^{-1}e_1.
  \end{align*}
  Comparing the entries of the column vectors on the two sides of the above identity gives the identities \eqref{eq:ogfi}.
  The identity \eqref{eq:ogf} is their sum.
\end{proof}
\begin{example}
  Figure~\ref{fig:two} shows the vertices within distance $3$ from the root of a tree with two lineal equivalence classes, labelled `$1$' and `$2$'.
  The vertices  of type $1$ have three children, one of type $1$ and two of type $2$, while the vertices of type $2$ have no children of type $1$ and two children of type $2$.
  \begin{figure}
    \begin{center}
      \begin{tikzpicture}[scale = 0.6]
        \path (10,0) node [draw, shape = circle, fill = lightgray] (P1) {$1$};
        \path (5,-2) node [draw, shape = circle, fill = lightgray] (P21) {$1$};
        \path (10,-2) node [draw, shape = circle] (P22) {$2$};
        \path (15,-2) node [draw, shape = circle] (P23) {$2$};
        \path (2,-4) node [draw, shape = circle, fill = lightgray] (P31) {$1$};
        \path (4.5,-4) node [draw, shape = circle] (P32) {$2$};
        \path (7,-4) node [draw, shape = circle] (P33) {$2$};
        \path (9.5,-4) node [draw, shape = circle] (P34) {$2$};
        \path (12,-4) node [draw, shape = circle] (P35) {$2$};
        \path (14.5,-4) node [draw, shape = circle] (P36) {$2$};
        \path (17,-4) node [draw, shape = circle] (P37) {$2$};
        \path (0,-6) node [draw, shape = circle, fill = lightgray] (P41) {$1$};
        \path (1.3,-6) node [draw, shape = circle] (P42) {$2$};
        \path (2.6,-6) node [draw, shape = circle] (P43) {$2$};        
        \path (3.9,-6) node [draw, shape = circle] (P44) {$2$};        
        \path (5.2,-6) node [draw, shape = circle] (P45) {$2$};        
        \path (6.5,-6) node [draw, shape = circle] (P46) {$2$};        
        \path (7.8,-6) node [draw, shape = circle] (P47) {$2$};        
        \path (9.1,-6) node [draw, shape = circle] (P48) {$2$};        
        \path (10.4,-6) node [draw, shape = circle] (P49) {$2$};        
        \path (11.7,-6) node [draw, shape = circle] (P4A) {$2$};        
        \path (13,-6) node [draw, shape = circle] (P4B) {$2$};        
        \path (14.3,-6) node [draw, shape = circle] (P4C) {$2$};        
        \path (15.6,-6) node [draw, shape = circle] (P4D) {$2$};        
        \path (16.9,-6) node [draw, shape = circle] (P4E) {$2$};        
        \path (18.2,-6) node [draw, shape = circle] (P4F) {$2$};        
        \draw[->, > = latex]
        (P1) edge (P21)
        (P1) edge (P22)
        (P1) edge (P23)
        (P21) edge (P31)
        (P21) edge (P32)
        (P21) edge (P33)
        (P22) edge (P34)
        (P22) edge (P35)
        (P23) edge (P36)
        (P23) edge (P37)
        (P31) edge (P41)
        (P31) edge (P42)
        (P31) edge (P43)
        (P32) edge (P44)
        (P32) edge (P45)
        (P33) edge (P46)
        (P33) edge (P47)
        (P34) edge (P48)
        (P34) edge (P49)
        (P35) edge (P4A)
        (P35) edge (P4B)
        (P36) edge (P4C)
        (P36) edge (P4D)
        (P37) edge (P4E)
        (P37) edge (P4F)
        ;
      \end{tikzpicture}
    \end{center}
    \caption{Part of a tree with two lineal isomorphism classes}
    \label{fig:two}
  \end{figure}
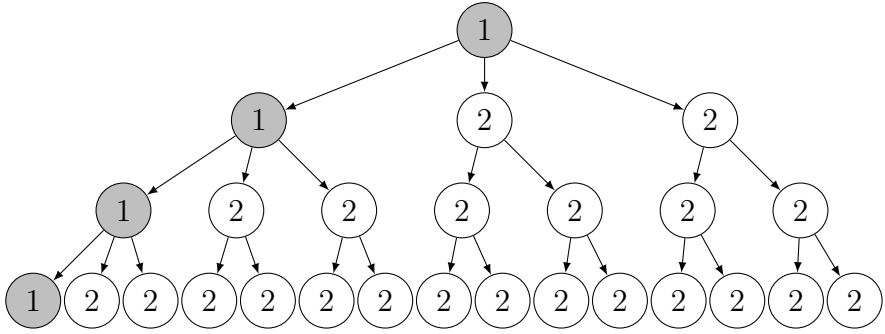
  The ``branching matrix'' is
  \begin{equation*}
    \B =
    \begin{pmatrix}
      1 & 0 \\
      2 & 2
    \end{pmatrix}.
  \end{equation*}
  Therefore, if $v^{(n)}_1$ and $v^{(n)}_2$ are the numbers of nodes of type $1$ and $2$ respectively which lie at distance $n$ from the root, then
  \begin{equation*}
    v^{(n)} :=
    \begin{pmatrix}
      v^{(n)}_1\\
      v^{(n)}_2
    \end{pmatrix}
    =
    \B^n
    \begin{pmatrix}
      1\\
      0
    \end{pmatrix}.
  \end{equation*}
  We have
  \begin{align*}
    \sum_{n=0}^\infty v^{(n)}t^n & = (I-\B t)^{-1}e_1\\
    & =
    \begin{pmatrix}
      \frac 1{1-t} & 0\\
      \frac{2t}{(1-t)(1-2t)} & \frac 1{1-2t}
    \end{pmatrix}
    \begin{pmatrix}
      1\\
      0
    \end{pmatrix}\\
    & =
    \begin{pmatrix}
      \frac 1{1-t}\\
      \frac{2t}{(1-t)(1-2t)}
    \end{pmatrix}.
  \end{align*}
  If $T_n$ is the set of nodes at a distance $n$ from the root, we have
  \begin{equation*}
    \sum_{n=0}^\infty |T_n|t^n = \frac 1{(1-t)(1-2t)}
  \end{equation*}
\end{example}

\section{Conjugacy classes of commuting tuples in groups}
\label{sec:conjugacy-classes}
Let $G$ be a finite group.
Then $G$ acts on $G^n$ for each non-negative integer $n$ by simultaneous conjugacy:
\begin{equation}
  \label{eq:simul-conj}
  g\cdot(x_1,\dotsc,x_n) = (gx_1g^{-1},\dotsc,gx_ng^{-1}).
\end{equation}
If $a_n$ is the number of orbits for the action of $G$ in $G^n$, it is not difficult to see that $f_G(t) = \sum_{n=0}^\infty a_nt^n$ is a rational function in $t$.
Indeed, by Burnside's lemma
\begin{equation*}
  a_n = \frac 1{|G|}\sum_{g\in G} |Z_G(g)|^n,
\end{equation*}
where $Z_G(g)$ denotes the centralizer of $g$ in $G$.
Therefore,
\begin{align*}
  \sum_{n=0}^\infty a_n t^n & = \sum_{n=0}^\infty \frac 1{|G|}\sum_{g\in G}|Z_G(g)|^nt^n\\
  & = \frac 1{|G|} \sum_{g\in G} \frac 1{1-|Z_G(g)|t},
\end{align*}
which is a rational function in $t$.

\begin{example}
  \label{example:fSn}
  Taking $G = S_m$ (the symmetric group on $m$ symbols),
  \begin{align*}
    f_{S_m}(t) & = \frac 1{m!} \sum_{w\in S_m} \frac 1{1-Z_{S_m}(w)t}\\
    & = \frac 1{m!}\sum_{\lambda\vdash m} \frac{m!}{z_\lambda} \frac 1{1-z_\lambda t}\\
    & = \sum_{\lambda\vdash m} \frac 1{z_\lambda(1-z_\lambda t)}.
  \end{align*}
  Here $\lambda\vdash m$ signifies that $\lambda$ is a partition of $m$, and for each such partition, $z_\lambda$ denotes the cardinality of the centralizer in $S_m$ of a permutation with cycle type $\lambda$.
  If, for each positive integer $i$,  $m_i$ is the number of occurrences of $i$ in $\lambda$, then
  \begin{equation*}
    z_\lambda = \prod_{i=1}^\infty m_i!i^{m_i}.
  \end{equation*}
  The values of $f_{S_m}$ for small values of $m$ are given in Table~\ref{table:fSn}.
  \begin{table}[h]
    \begin{equation*}
      \def\arraystretch{2.2}
      \begin{array}{cc}
        \hline
        m & f_{S_m}(t)\\
        \hline
        1 & \dfrac 1{1-t}\\
        2 & \dfrac 1{1-2t}\\
        3 & \dfrac{1 - 8t + 14t^2}{(1-2t)(1-3t)(1-6t)}\\
        4 & \dfrac{1 - 34t + 276 t^2 -584 t^3}{(1-3t)(1-4t)(1-8t)(1-24t)}\\
        5 & \dfrac{1 - 148t + 3746 t^2 - 36984 t^3 + 159200t^4 - 249792 t^5}{(1-4t)(1-5t)(1-6t)(1-8t)(1-12t)(1-120t)}\\
        \hline
      \end{array}
    \end{equation*}
    \caption{Generating functions for simultaneous conjugacy classes in $S_n$.}
    \label{table:fSn}
  \end{table}
\end{example}
A more subtle problem is that of counting simultaneous conjugacy classes of $n$-tuples of commuting elements in $G$.
For each $n\geq 0$, let
\begin{equation*}
  G^{(n)} = \{(g_1,\dotsc,g_n)\in G^n\mid g_ig_j = g_jg_i \text{ for all }1\leq i,j\leq n\}.
\end{equation*}
In particular, $G^{(0)}$ is the trivial group.

Let $c_n$ denote the number of orbits for the action of $G$ on $G^{(n)}$ by simultaneous conjugation, as given in (\ref{eq:simul-conj}).
Consider the generating function
\begin{equation*}
  h_G(t) = \sum_{n=0}^\infty c_n t^n.
\end{equation*}
Because the elements $g_1,g_2,\dotsc,g_n$ are no longer independent, Burnside's lemma can no longer be used to prove the rationality of $h_G(t)$.
However, Theorem~\ref{theorem:lineal} allows us to show that $h_G(t)$ is rational in $t$, and gives an algorithm to compute it for any finite group.
\begin{theorem}
  \label{theorem:group-tuples}
  For every finite group $G$, the formal power series $h_G(t)$ defined above is a rational function of $t$.
\end{theorem}
\begin{proof}
  Let $T_n^G$ denote the set of $G$-orbits in $G^{(n)}$.
  Say that $Y\in T_{n+1}^G$ is connected to $X\in T_n$ by an edge if there exists $(g_1,\dotsc,g_{n+1})\in Y$ such that $(g_1,\dotsc, g_n)\in X$.
  This gives $T^G = \coprod_{n=0}^\infty T^G_n$ the structure of a rooted tree with root $X_0$ being the unique element of $T_0^G$.
  For $(g_1,\dotsc,g_n)\in G^{(n)}$, let
  \begin{equation*}
    Z_G(g_1,\dotsc,g_n) = Z_G(g_1)\cap Z_G(g_2) \cap \dotsb \cap Z_G(g_n).
  \end{equation*}

  We will see in Theorem~\ref{theorem:cent} below that the $G$-orbit of $(g_1,\dotsc,g_n)$ is lineally isomorphic to the $G$-orbit of $(s_1,\dotsc, s_l)$ in $T^G$ if the group $Z_G(g_1,\dotsc,g_n)$ is isomorphic to the group $Z_G(s_1,\dotsc,s_l)$.
  Since each of these centralizers is a subgroup of the finite group $G$, there are only finitely many possible isomorphism classes for them, and so only finitely many lineal isomorphism classes in $T^G$.
  Thus Theorem~\ref{theorem:lineal} applies, and $h_G(t)$ is a rational function of $t$.
\end{proof}
We now come to Theorem~\ref{theorem:cent} and its proof (which will complete the proof of Theorem~\ref{theorem:group-tuples}).
\begin{theorem}
  \label{theorem:cent}
  Suppose that $(g_1,\dotsc,g_n)\in G^{(n)}$ lies in the $G$-orbit $X\in T_n^G$, then the full subtree of $T^G(X)$ (rooted at $X$) consisting of descendants of $X$ in $T^G$ is isomorphic to the rooted tree $T^{Z_G(g_1,\dotsc,g_n)}$ associated to the group $Z_G(g_1,\dotsc,g_n)$.
\end{theorem}
\begin{proof}
  Let $S = Z_G(g_1,\dotsc,g_n)$.
  Define a map $S^{(l)}\to G^{(n+l)}$ by
  \begin{equation*}
    (s_1,\dotsc,s_l)\mapsto (g_1,\dotsc,g_n,s_1,\dotsc,s_l).
  \end{equation*}
  It is easy to check that this map induces an isomorphism of rooted trees $T^S\to T^G(X)$.
\end{proof}
We now consider the examples of symmetric groups: since $S_2$ is abelian, $h_{S_2}(t) = f_{S_2}(t)$, which was computed earlier in this section.
Now $S_3$ has three conjugacy classes, which lie in different lineal isomorphism classes in $T^{S_3}$ since they have non-isomorphic centralizers.
The centralizers are given in the following table:
\begin{equation*}
  \begin{array}{ll}
    \hline
    \lambda & Z_{S_3}(x)\\
    \hline
    (1,1,1) & S_3\\
    (2,1) & C_2\\
    (3) & C_3\\
    \hline
  \end{array}
\end{equation*}
With the exception of the class of the identity element (with cycle type $(1,1,1)$) each of these centralizers is abelian.
If the orbit of a tuple has abelian centralizer, then all its descendants are lineally isomorphic to it.
For the singleton orbit of the identity element, we once again have three children, one corresponding to each partition of $3$.
Thus every pair of commuting elements on $S_3$ has centralizer isomorphic to that of an element of $S_3$.
The branching matrix of Theorem~\ref{theorem:lineal} is
\begin{equation*}
  \B =
  \begin{pmatrix}
    1 & 0 & 0\\
    1 & 2 & 0\\
    1 & 0 & 3
  \end{pmatrix}
\end{equation*}
A routine calculation shows that
\begin{equation*}
  h_{S_3}(t) = \frac{1 - 3t + t^2}{(1-t)(1-2t)(1-3t)}.
\end{equation*}

The group $S_4$ five conjugacy classes with centralizers given by:
\begin{equation*}
  \begin{array}{ll}
    \hline
    \lambda & Z_{S_4}(x)\\
    \hline
    (1,1,1,1) & S_4\\
    (2, 1, 1) & C_2\times C_2\\
    (2, 2) & C_2\wr S_2\\
    (3, 1) & C_3\\
    (4) & C_4\\
    \hline
  \end{array}
\end{equation*}
The only troublesome case here is $\lambda = (2,2)$.
The centralizer group in this case is a non-abelian group of order $8$, which we now proceed to analyse: for concreteness, consider the centralizer of the permutation $(12)(34)\in S_4$.
The centralizer subgroup consists of the permutations:
\begin{equation*}
  H = \{1, (12)(34), (12), (34), (13)(24), (14)(23), (1423), (1324)\}.
\end{equation*}
This group has five conjugacy classes, with centralizers given by:
\begin{equation*}
  \begin{array}{ll}
    \hline
    \text{class} & \text{centralizer}\\
    \hline
    1 & C_2\wr S_2\\
    (12)(34) & C_2\wr S_2\\
    (12), (34) & C_2\times C_2\\
    (14)(23), (13)(24) & C_2\times C_2\\
    (1324), (1423) & C_4\\
    \hline
  \end{array}
\end{equation*}
Thus, as in the case of $S_3$, the centralizers of pairs are all centralizers of elements in $S_4$.
The branching matrix is
\begin{equation*}
  \B =
  \begin{pmatrix}
    1 & 0 & 0 & 0 & 0\\
    1 & 4 & 2 & 0 & 0\\
    1 & 0 & 2 & 0 & 0\\
    1 & 0 & 0 & 3 & 0\\
    1 & 0 & 1 & 0 & 4
  \end{pmatrix},
\end{equation*}
which gives:
\begin{equation*}
  h_{S_4}(t) = \frac{1 - 5t + 6t^2 - t^3}{(1-t)(1-2t)(1-3t)(1-4t)}.
\end{equation*}
Similarly for $S_5$ we have
\begin{equation*}
  \begin{array}{ll}
    \hline
    \lambda & Z_{S_5}(x)\\
    \hline
    (1^5) & S_5\\
    (2,1^3) & C_2\times S_3\\
    (2,2,1) & C_2\wr S_2\\
    (3,1,1) & C_3\times S_2\\
    (3,2) & C_3\times C_2\\
    (4,1) & C_4\\
    (5) & C_5\\
    \hline
  \end{array}
\end{equation*}
The classes corresponding to the partitions $(3,1,1)$ and $(3,2)$ can be clubbed, as they have isomorphic centralizers (and therefore are lineally isomorphic).
In $C_2\times S_3$, we have the following count of classes and their centralizers:
\begin{equation*}
  \begin{array}{lc}
    \hline
    \text{centralizer} & \text{no. of classes}\\
    \hline
    C_2\times S_3 & 2\\
    C_2\times C_2 & 2\\
    C_2\times C_3 & 2\\
    \hline
  \end{array}
\end{equation*}
The only centralizer here which is not the centralizer of an element of $S_5$ is $C_2\times C_2$, which is abelian.
The branching matrix is given by:
\begin{equation*}
  \B =
  \begin{pmatrix}
    1 & 0 & 0 & 0 & 0 & 0 & 0 \\
    1 & 2 & 0 & 0 & 0 & 0 & 0 \\
    1 & 0 & 2 & 0 & 0 & 0 & 0 \\
    2 & 2 & 0 & 6 & 0 & 0 & 0 \\
    1 & 0 & 1 & 0 & 4 & 0 & 0 \\
    1 & 0 & 0 & 0 & 0 & 5 & 0 \\
    0 & 2 & 2 & 0 & 0 & 0 & 4
  \end{pmatrix}
\end{equation*}
whence one may compute:
\begin{equation*}
  h_{S_5}(t) = \frac{1 - 11t + 34t^2 - 21t^3 + 2t^4}{(1-t)(1-2t)(1-4t)(1-5t)(1-6t)}.
\end{equation*}
Our calculations of $h_{S_m}$ for small values of $m$ are summarized in Table~\ref{tab:hSn}.
\begin{table}[h]
  \centering
  \begin{equation*}
    \def\arraystretch{2.2}
    \begin{array}{cc}
      \hline
      m & h_{S_m}(t)\\
      \hline
      1 & \dfrac 1{1-t}\\
      2 & \dfrac 1{1-2t}\\
      3 & \dfrac{1-3t+t^2}{(1-t)(1-2t)(1-3t)}\\
      4 & \dfrac{1-5t+6t^2-t^3}{(1-t)(1-2t)(1-3t)(1-4t)}\\
      5 & \dfrac{1 - 11t + 34t^2 - 21t^3 + 2t^4}{(1-t)(1-2t)(1-4t)(1-5t)(1-6t)}\\
      \hline
    \end{array}
  \end{equation*}
  \caption{Generating functions for simultaneous conjugacy classes of commuting elements in $S_m$.}
  \label{tab:hSn}
\end{table}
The techniques at hand are not strong enough to derive an analog of the formula in Example~\ref{example:fSn} for $h_{S_m}$ for general $m$.

Theorem~\ref{theorem:group-tuples} can also be stated for finite algebras:
\begin{theorem}
  \label{theorem:alg-tuples}
  Let $A$ be a finite ring, $A^*$ its multiplicative group of units, and let
  \begin{equation*}
    A^{(n)} = \{(a_1,\dotsc,a_n)\in A^n\mid a_ia_j=a_ja_i \text{ for } 1\leq i,j\leq n\}.
  \end{equation*}
  Then $A^*$ acts on $A^{(n)}$ by simultaneous similarity:
  \begin{equation*}
    u\cdot (a_1,\dotsc,a_n) = (ua_1u^{-1},\dotsc,ua_nu^{-1}).
  \end{equation*}
  Let $c_A(n)$ denote the number of orbits for the action of $A^*$ on $A^{(n)}$.
  Then the generating function
  \begin{equation*}
    h_A(t) = \sum_{n=0}^\infty c_A(n)t^n
  \end{equation*}
  is a rational function of $t$.
\end{theorem}
\begin{proof}
  The proof is similar to Theorem~\ref{theorem:group-tuples}: Let
  \begin{equation*}
    Z_A(a_1,\dotsc,a_n) = Z_A(a_1)\cap \dotsc Z_A(a_n),
  \end{equation*}
  where $Z_A(a)$ denotes the subring of elements of $A$ that commute with $a$.
  The $A^*$-orbits of $(a_1,\dotsc,a_n)$ and $(b_1,\dotsc,b_l)$ are lineally isomorphic if the rings $Z(a_1,\dotsc,a_n)$ and $Z(b_1,\dotsc,b_l)$ are isomorphic.
\end{proof}
Let $\Fq$ be a finite field of order $q$.
Taking $A$ to be the algebra $M_m(\Fq)$ of $m\times m$ matrices with entries in $\Fq$ in Theorem~\ref{theorem:alg-tuples} gives simultaneous similarity classes of commuting $n$-tuples in $M_m(\Fq)$.
An $n$-tuple of commuting matrices is nothing but an $\Fq[x_1,\dotsc,x_n]$-module.
Two modules are isomorphic if and only if the corresponding $n$-tuples are simultaneously similar.
So we have:
\begin{theorem}
  \label{theorem:polyalg}
  Let $\Fq$ denote the finite field with $q$ elements, and for each positive integer $m$ let $c_{q,m}(n)$ denote the number of isomorphism classes of $m$-dimensional modules for the polynomial algebra $\Fq[x_1,\dotsc,x_n]$.
  Then the generating function:
  \begin{equation*}
    h_{q,m}(t) = \sum_{n=0}^\infty c_{q,m}(n)t^n
  \end{equation*}
  is a rational function of $t$.
\end{theorem}

The polynomials $h_{q,m}(t)$ are quite difficult to compute for $m\geq 4$, but seem to have very interesting combinatorial properties, an investigation of which is the subject of \cite{uday}.
For example,
\begin{align*}
  h_{q,1}(t) &= \frac 1{1-qt}\\
  h_{q,2}(t) &= \frac 1{(1-qt)(1-q^2t)}\\
  h_{q,3}(t) &= \frac {1+q^2t^2}{(1-qt)(q-q^2t)(1-q^3t)}.
\end{align*}
The details of these (and further) calculations can be found in \cite{uday}.

As with groups, counting of $A^*$-orbits in $A^n$ (instead of $A^{(n)}$) is much easier, because of the applicability of Burnside's lemma.
When $A = M_m(\Fq)$ this becomes the problem of counting isomorphism classes of $m$-dimensional representations of the free algebra $\Fq\langle x_1,\dotsc,x_n\rangle$, which in turn is a special case of the problem of counting representations of a quiver with the fixed dimension vector, a well-developed program which was started by Kac \cite{Kac83} in 1983 and culminated in the recent work of Hausel, Letellier and Rodriguez-Villegas \cite{HLR13} and Mozgovoy \cite{Moz}.
In contrast, we do not even know that $c_{q,m}(n)$ is a polynomial in $q$ for $m>4$.

The counting of isomorphism classes of $\Fq[x_1,\dotsc,x_n]$-modules appears to be related to the counting of similarity classes of matrices in finite quotients of discrete valuation rings.
Let $R$ be a discrete valuation ring with residue field $\Fq$.
Let $P$ denote the maximal ideal of $R$.
Matrices $A,B\in M_m(R/P^n)$ are said to be similar if $B = gAg^{-1}$ for some $g\in GL_m(R/P^n)$.
From the work of Singla~\cite{pooja10}, Jambor and Plesken~\cite{JamborPlesken} and Prasad, Singla and Spallone~\cite{PSS}, we know that the number of isomorphism classes of $m$-dimensional $\Fq[x_1,x_2]$-modules (what we have called $c_{q,m}(2)$ above) is equal to the  number of similarity classes in $M_m(R/P^2)$.
Further, by comparing the values for $h_{q,m}(t)$ quoted above with the results obtained by Avni, Onn, Prasad and Vaserstein~\cite{AOPV} we find that for $m\leq 3$, $c_{q,m}(n)$ matches the number of similarity classes in $M_m(R/P^n)$ for $m\leq 3$ and all $n$.

One is led to the following conjecture:
\begin{conjecture*}
  The number of similarity classes in $M_m(R/P^n)$ is equal to the number of isomorphism classes of $m$-dimensional $\Fq[x_1,\dotsc,x_n]$-modules for all positive integers $m$ and $n$.
\end{conjecture*}
If $C_{q,m}(n)$ denotes the number of similarity classes in $M_m(R/P^n)$, then du Sautoy~\cite{duS} has shown (using model theory) that when $R$ has characteristic zero, then $C_{q,m}(n)$, as a sequence in $n$, has a rational generating function.
His result is the analog of Theorem~\ref{theorem:polyalg} for similarity classes in $M_m(R/P^n)$.

\section{Point and vector configurations}
\label{sec:point-vec-configs}
The symmetric group $S_m$ acts on the set $\m = \{1,\dotsc,m\}$.
An $n$-point configuration in $\m$ is, by definition, an orbit of $S_m$ for its action on $\m^n$ by
\begin{equation*}
  w\cdot(x_1,\dotsc,x_n) = (w\cdot x_1,\dotsc, w\cdot x_n).
\end{equation*}
For example, there are two $2$-point configurations in $\m$ for $m\geq 2$: either the points $x_1$ and $x_2$ coincide, or they are distinct.
Likewise, there are five $3$-point configurations in $\m$ for $m\geq 3$, represented by
\begin{equation*}
  (1, 1, 1), (1, 1, 2), (1, 2, 1), (2, 1, 1)\text{ and } (1, 2, 3).
\end{equation*}
Let $c_m(n)$ denote the number of $n$-point configurations in $\m$.

We may compute $c_m(n)$ using Burnside's lemma.
With the notation of Example~\ref{example:fSn},
\begin{align*}
  c_m(n) & = \frac 1{m!} \sum_{w\in S_m} (\text{no. of fixed points of $w$})^n\\
  & = \sum_{\lambda\vdash m} \frac {m_1(\lambda)^n}{z_\lambda},
\end{align*}
so that
\begin{equation*}
  \sum_{n=0}^\infty c_m(n)t^n = \sum_{\lambda\vdash m} \frac 1{z_\lambda} \frac 1{1-m_1(\lambda)t}.
\end{equation*}

However, we shall see below that using Theorem~\ref{theorem:lineal} for the same computation leads to the standard ordinary generating functions and recurrence relations for Bell numbers \cite[A000110]{oeis} and Stirling numbers of the second kind \cite[A008277]{oeis}.

Let $T_n^{(m)}$ denote the set of $S_m$ orbits in $\m^n$, and $T^{(m)} = \coprod_{n=0}^\infty T_n^{(m)}$.
Say that $Y\in T_{n+1}^{(m)}$ is a child of $X\in T_n^{(m)}$ if there exists $(x_1,\dotsc,x_{n+1})\in Y$ such that $(x_1,\dotsc,x_n)\in X$.
We say that $X\in T^{(m)}_n$ has type $i$ if, for any $(x_1,\dotsc,x_n)\in X$, the number of distinct elements in the set $\{x_1,\dotsc,x_n\}$ is $i$.
Clearly, if $X$ has type $i$, then each of its children has type either $i$ or $i+1$.
Also, $(x_1,\dotsc,x_n,x_{n+1})$ and $(x_1,\dotsc,x_n,x_{n+1}')$ lie in the same $S_m$-orbit if and only if there exists a permutation which fixes $x_1,\dotsc,x_n$ and maps $x_{n+1}$ to $x_{n+1}'$.

Now suppose that $X$ is of type $i$ and $(x_1,\dotsc,x_n)\in X$.
If, for some $x_{n+1}\in \m$, the orbit of $(x_1,\dotsc,x_{n+1})$ is also of type $i$, then $x_{n+1}$ coincides with one of $x_1,\dotsc,x_n$ and is therefore fixed by any $w\in S_m$ which fixes them.
Thus, a node of type $i$ has $i$ children of type $i$.
On the other hand, if the orbit of $(x_1,\dotsc,x_{n+1})$ is of type $i+1$, then $x_{n+1}$ is different from each of $x_1,\dotsc,x_n$ and can therefore be permuted to any other element of $\m$ that is distinct from $x_1,\dotsc,x_n$ while fixing them.
It follows that a node of type $i$ has $1$ child of type $i+1$.

The branching matrix is given by
\begin{equation}
  \label{eq:branching-point-conf}
  \B =
  \begin{pmatrix}
    0 &   &   &   &  &   &  \\
    1 & 1 &   &   &  &   &  \\
    & 1 & 2 &   &  &   &  \\
    &   & 1 & 3 &  &   &  \\
    &  &  &  & \ddots &   &  \\
    &   &   &   &  & m-1 &  \\
    &   &   &   &  & 1 & m
  \end{pmatrix},
\end{equation}
a matrix whose diagonal entries are $0,\dotsc,m$, with $1$'s just below the diagonal and with all other entries zero.
One easily computes
\begin{equation}
  \label{eq:typei}
  e_i(1-\B t)^{-1}e_1 = \prod_{r=1}^i \frac t{1-rt}.
\end{equation}
We obtain:
\begin{theorem}
  The sequence $c_m(n)$ has generating function
  \label{theorem:point-configs}
  \begin{equation*}
    \sum_{n=0}^\infty c_m(n)t^n = \sum_{i=0}^m \prod_{r=0}^i \frac t{1-rt}
  \end{equation*}
\end{theorem}
Each $n$-tuple $(x_1,\dotsc,x_n)$ gives rise to an equivalence relation on $\n$; indices $i,j\in \n$ are equivalent if $x_i=x_j$.
Two tuples are in the same $S_m$-orbit if and only if they give rise to the same equivalence relation on $\n$.
Thus the number of $n$-point configurations in $\m$ is nothing but the number of equivalence relations on $\n$ with at most $m$ equivalence classes.
For $m\geq n$, this number is the well-known Bell number $B_n$.
Under the correspondence between $n$-point configurations in $\m$ and equivalence relations on $\n$ with at most $m$ equivalence classes, point configurations of type $i$ map to equivalence relations with exactly $i$ equivalence classes.
The number of equivalence relations with exactly $i$ equivalence classes is the well-known \emph{Stirling number of the second kind}, usually denoted $S(n,i)$ or $\begin{Bmatrix} n\\i\end{Bmatrix}$ \cite[Section~1.9]{ec1}.
The identity~\eqref{eq:typei} becomes a well-known generating function for Stirling numbers of the second kind \cite[Eq. (1.94c)]{ec1}:
\begin{equation*}
  \sum_{n=0}^\infty S(n,i)t^n = t^i\prod_{r=0}^i \frac 1{1-rt}.
\end{equation*}
This identity reflects the obvious fact that $S(n,i) = 0$ for $i>n$.
The branching rule \eqref{eq:branching-point-conf} gives the standard recurrence relation for Stirling numbers of the second kind \cite[Eq.~(1.93)]{ec1}:
\begin{equation}
  \label{eq:recurrence-Sterling2}
  S(n,k) = S(n-1,k-1) + k S(n-1,k).
\end{equation}

By definition, the Bell number $B_n$ is the number of equivalence relations on a set of order $n$.
Clearly, $B_n = \sum_{i=0}^\infty S(n,i)$ which equals $c_m(n)$ provided that $m\geq n$.
Thus the ordinary generating function for Bell numbers (see, e.g., \cite[Lemma~8]{LangA071919}) is obtained:
\begin{equation*}
  \sum_{n=0}^\infty B_nt^n = \sum_{i=0}^\infty \prod_{r=0}^i \frac t{1-rt}.
\end{equation*}

Let $\Fq$ denote a finite field with $q$ elements.
The general linear group $GL_m(\Fq)$ acts on the vector space $\Fq^m$, and therefore also on $n$-tuples of vectors in it:
\begin{equation*}
  g\cdot (x_1,\dotsc,x_n) = (g(x_1),\dotsc,g(x_n)),
\end{equation*}
for $g\in GL_m(\Fq)$ and for $x_1,\dotsc, x_n\in \Fq^m$.
A configuration of $n$ vectors in $\Fq^m$ is an orbit of $GL_m(\Fq)$ on $(\Fq^m)^n$.

Let $T^{q,m}_n$ denote the set of $GL_m(\Fq)$-orbits in $(\Fq^m)^n$.
Say that $Y\in T^{q,m}_{n+1}$ is a child of $X\in T^{q,m}_{n}$ if there exists $(x_1,\dotsc,x_{n+1})\in Y$ such that $(x_1,\dotsc,x_n)\in X$.
Let $T^{q,m} = \coprod_{n=0}^\infty T^{q,m}_n$.
We say that $X\in T_{q,m}$ has type $i$ if, for any $(x_1,\dotsc,x_n)\in X$, the dimension of the subspace spanned by the set $\{x_1,\dotsc,x_n\}$ is $i$.
If $X$ is of type $i$, then a child of $X$ must be of type $i$ or $i+1$.
If $(x_1,\dotsc,x_n,x_{n+1})$ is of type $i$, then $x_{n+1}$ lies in the span of $x_1,\dotsc,x_n$.
Therefore any element of $GL_m(\Fq)$ that fixes $x_1,\dotsc,x_n$ fixes $x_{n+1}$ as well.
Therefore, a tuple of type $i$ has $q^i$ children of type $i$.
If, on the other hand, $(x_1,\dotsc,x_{n+1})$ and $(x_1,\dotsc,x'_{n+1})$ both have type $i+1$, then $x_{n+1}$ and $x'_{n+1}$ are linearly independent of $x_1,\dotsc,x_n$, so there exists $g\in GL_m(\Fq)$ mapping $x_{n+1}$ to $x'_{n+1}$ while fixing $x_1,\dotsc,x_n$.
Therefore, a tuple of type $i$ has only one child of type $i+1$.

The branching matrix is given by
\begin{equation}
  \label{eq:branching-vect-configs}
  \B =
  \begin{pmatrix}
    1 &   &   &   &  &   &  \\
    1 & q &   &   &  &   &  \\
    & 1 & q^2 &   &  &   &  \\
    &   & 1 & q^3 &  &   &  \\
    &  &  &  & \ddots &   &  \\
    &   &   &   &  & q^{m-1} &  \\
    &   &   &   &  & 1 & q^m
  \end{pmatrix}.
\end{equation}
We obtain:
\begin{equation}
  \label{eq:eii-bt}
  e_i(I-\B t)^{-1}e_1 = t^i \prod_{r=0}^i\frac 1{1-q^rt} \text{ for } 1\leq i \leq m.
\end{equation}
\begin{theorem}
  \label{theorem:vs-configs}
  The generating function fot $\{|T^{q,m}_n|\}_n$ is
  \begin{equation*}
    \sum_{n=0}^\infty |T^{q,m}_n|t^n = \sum_{i=0}^m t^i \prod_{r=0}^i \frac 1{1-q^rt}
  \end{equation*}
\end{theorem}
The quantity $|T^{q,m}_n|$ does not depend on $m$ so long as $m\geq n$.
The stable value of this quantity $B_{q,n}:=|T^{q,n}_n|$ may be regarded as an analog of the Bell number for which we get the ordinary generating function:
\begin{equation}
  \label{eq:ogfqbell}
  \sum_{n=0}^\infty B_{q,n} t^n = \sum_{i=0}^\infty t^i \prod_{r=0}^i \frac 1{1-q^rt}.
\end{equation}
Likewise, if we only count those $n$-vector configurations in $\Fq^m$ which span an $i$-dimensional subspace of $\Fq^m$ when $m\geq i$ (given by \eqref{eq:eii-bt}, these clearly do not depend on $m$ so long as $m\geq i$) we obtain an analog of the Stirling number of the second kind:
\begin{equation*}
  \sum_{n=0}^\infty S_q(n,i)t^n = t^i\prod_{r=0}^i \frac 1{1-q^rt}.
\end{equation*}
The branching rule \eqref{eq:branching-vect-configs} gives the recurrence relation
\begin{equation}
  \label{eq:q-Stirling-rec}
  S_q(n,i) = S_q(n-1,i-1) + q^i S_q(n-1,i)\text{ for } 0<i<n.
\end{equation}
which is not the same as the recurrence relation for the usual $q$-Stirling numbers \cite[Eq.~(3.8)]{Gould}; it does not specialize to \eqref{eq:recurrence-Sterling2} at $q=1$.
Instead it is one of the Pascal identities for Gaussian binomial coefficients:
\begin{equation*}
  \binom ni_q = \binom{n-1}{i-1}_q + q^i \binom{n-1}i_q.
\end{equation*}
Also, just like Gaussian binomial coefficients, $S_q(n,0) =S_q(n,n) =1$, so we have:
\begin{theorem}
  \label{theorem:subspaces-configs}
  The number of $n$-vector configurations in $\Fq^m$ whose span has dimension $i$ when $m\geq i$ is equal to the number of $i$-dimensional subspaces of an $n$-dimensional vector space:
  \begin{equation}
    \label{eq:1}
    S_q(n,i) = \binom ni_q.
  \end{equation}
\end{theorem}
The identity \eqref{eq:q-Stirling-rec} becomes
\begin{equation*}
  \sum_{n=0}^\infty \binom ni_q t^n = t^i \prod_{r=0}^i \frac 1{1-q^it},
\end{equation*}
a well-known generating function for Gaussian binomial coefficients \cite[p.~74]{ec1}.

For a bijective proof of Theorem~\ref{theorem:subspaces-configs}, given an $n$-tuple $(x_1,\dotsc,x_n)$ of vectors in $\Fq^m$, write a matrix $\X$ whose columns are the coordinates of these vectors; if $x_j = (x_{1j},x_{2j},\dotsc,x_{mj})$, then
\begin{equation*}
  \X =
  \begin{pmatrix}
    x_{11} & x_{12} & \cdots & x_{1n}\\
    x_{21} & x_{22} & \cdots & x_{2n}\\
    \vdots & \vdots & \ddots & \vdots\\
    x_{m1} & x_{m2} & \cdots  & x_{mn}
  \end{pmatrix}.
\end{equation*}
Similarly, if $(x'_1,\dotsc,x'_n)$ is another $n$-tuple of vectors, and $\X'$ is the corresponding matrix, then $(x_1,\dotsc,x_n)$ and $(x_1',\dotsc,x_n')$ lie in the same $GL_m(\Fq)$-orbit if and only if there exists $g\in GL_m(\Fq)$ such that
\begin{equation}
  \label{eq:equiv}
  \X' = g\X.
\end{equation}
However, the condition \eqref{eq:equiv} is also necessary and sufficient for two $m\times n$ matrices $\X$ and $\X'$ to have the same row space.
Thus the function that takes the tuple $(x_1,\dotsc,x_n)$ to the row space of $\X$ descends to a bijective map from the set of $n$-vector configurations in $\Fq^m$ to the set of subspaces of $\Fq^n$ which are spanned by $m$ vectors, in other words, subspaces of dimension $m$ or less.
If $m\geq n$, then this is the set of all subspaces of $\Fq^n$.
Since the row rank of a matrix is equal to its column rank, if the $n$-tuple $(x_1,\dotsc, x_n)$ spans an $i$-dimensional vector space, the row space of $\X$ is an $i$-dimensional subspace of $\Fq^m$.
This sets up a bijection from the set of configurations of $n$ vectors in $\Fq^m$ of type $i$ and the set of all $i$-dimensional subspaces of $\Fq^m$ when $m\geq i$.

Anilkumar and Prasad \cite{anilprasad} studied the number of configurations of pairs in finite abelian $p$-groups.
They conjectured that these numbers are represented by polynomials in $p$ with non-negative integer coefficients.
It would be interesting to try to generalize the ideas behind Theorems~\ref{theorem:vs-configs} and~\ref{theorem:subspaces-configs} to counting configurations of tuples in finite abelian $p$-groups.
\subsection*{Acknowledgements}
The author thanks Uday Bhaskar Sharma for many helpful discussions.
He also thanks Kunal Dutta and S. Viswanath for feedback on a draft of this paper.
\bibliographystyle{abbrv}
\bibliography{lineal}
\end{document}